\documentclass{amsart}

\usepackage{amssymb}
\usepackage{amsmath}
\usepackage{url}
\usepackage{enumerate}

\usepackage{hyperref}
\hypersetup{
    colorlinks  =   true,
    linkcolor   =   blue,
    filecolor   =   blue,
    urlcolor    =   blue,
    citecolor   =   blue,
}

\newtheorem{theorem}{Theorem}[section]

\newtheorem{proposition}[theorem]{Proposition}
\newtheorem{corollary}[theorem]{Corollary}

\newtheorem{_definition}[theorem]{Definition}
\newenvironment{definition}{\begin{_definition}\rm}{\end{_definition}}

\newtheorem{_remark}[theorem]{\it Remark}
\newenvironment{remark}{\begin{_remark}\rm}{\end{_remark}}

\newtheorem{_example}[theorem]{Example}
\newenvironment{example}{\begin{_example}\rm}{\end{_example}}

\numberwithin{equation}{section}
\numberwithin{table}{section}
\numberwithin{figure}{section}

\newcommand{\F}{\mathord{\mathbb F}}
\newcommand{\Q}{\mathord{\mathbb  Q}}
\newcommand{\R}{\mathord{\mathbb R}}
\newcommand{\Z}{\mathord{\mathbb Z}}

\newcommand{\FFF}{\mathord{\mathcal F}}
\newcommand{\GGG}{\mathord{\mathcal G}}
\newcommand{\LLL}{\mathord{\mathcal L}}
\newcommand{\PPP}{\mathord{\mathcal P}}
\newcommand{\QQQ}{\mathord{\mathcal Q}}
\newcommand{\RRR}{\mathord{\mathcal R}}
\newcommand{\VVV}{\mathord{\mathcal V}}

\newcommand{\inj}{\hookrightarrow}
\newcommand{\isom}{\xrightarrow{\sim}}

\newcommand{\set}[2]{\{ {#1}\mid {#2}  \}}
\newcommand{\shortset}[2]{\{ {#1} \mid  {#2}   \}}
\newcommand{\gen}[1]{\langle {#1}  \rangle}

\newcommand{\tensor}{\otimes}
\newcommand{\sprime}{\sp\prime}
\newcommand{\spar}[1]{\sp{(#1)}}
\newcommand{\sperp}{\sp{\perp}}
\newcommand{\dual}{\sp{\vee}}
\newcommand{\semidirectproduct}{\rtimes}
\newcommand{\inv}{\sp{-1}}
\newcommand{\bdr}{\partial\,}
\newcommand{\Hom}{\mathord{\mathrm{Hom}}}
\newcommand{\OG}{\mathord{\mathrm{O}}}
\newcommand{\Aut}{\operatorname{\mathrm{Aut}}\nolimits}
\newcommand{\rank}{\operatorname{\mathrm{rank}}\nolimits}
\newcommand{\closure}[1]{\overline{#1}}

\newcommand{\mystruth}[1]{\phantom{\hbox{\vrule height #1}}}

\newcommand{\weyl}{\mathord{\mathbf w}}
\newcommand{\ADE}{\mathord{\rm ADE}}

\newcommand{\intf}[1]{\langle #1 \rangle}
\newcommand{\intfS}[1]{\langle #1\rangle_{\SS}}
\newcommand{\intfR}[1]{\langle #1\rangle_{R}}
\newcommand{\intfQ}[1]{\langle #1\rangle_{Q}}
\newcommand{\intfL}[1]{\langle #1\rangle_{10}}
\newcommand{\intfLL}[1]{\langle #1\rangle_{26}}

\newcommand{\OGP}[1]{\OG(#1, \PPP)}
\newcommand{\oLtenFtwo}{{2}^{21} \cdot {3}^{5} \cdot {5}^{2} \cdot {7} \cdot {17} \cdot {31} }
\newcommand{\Km}{\mathord{\rm Km}}
\newcommand{\Jac}{\mathord{\rm Jac}}

\renewcommand{\SS}{\mathbf{S}}
\renewcommand{\qed}{\hfill {$\Box$}}

\begin{document}
\title[Borcherds' method for Enriques surfaces]
{Borcherds' method for Enriques surfaces}
\author{Simon Brandhorst}
\address{(S. B.) Saarland University, Fachbereich Mathematik, Postfach 151150, Saarbr\"ucken, Germany}
\email{brandhorst@math.uni-sb.de}
\author{Ichiro Shimada}
\address{(I. S.) Department of Mathematics,
Graduate School of Science,
Hiroshima University,
1-3-1 Kagamiyama,
Higashi-Hiroshima,
739-8526 JAPAN}
\email{ichiro-shimada@hiroshima-u.ac.jp}

\thanks{The first author was supported by SFB-TRR 195.
The second author was supported by JSPS KAKENHI Grant Number 15H05738, ~16H03926,  and~16K13749}

\begin{abstract}
We classify all primitive embeddings
of  the lattice of numerical equivalence classes of divisors
of an Enriques surface with the intersection form multiplied by $2$
into an even unimodular hyperbolic lattice of rank $26$.
These embeddings have a property
that facilitates  the computation
of the automorphism group of an Enriques surface
by Borcherds' method.
\end{abstract}

\maketitle

\section{Introduction}\label{sec:intro}
Lattice theory is a
very strong tool in the study of $K3$ and Enriques surfaces.
Let $Z$ be a $K3$  or an Enriques surface defined over an algebraically closed field.
We denote by  $S_Z$  the lattice of numerical equivalence classes of
divisors on $Z$.
Note that $S_Z$ is even, and if $\rank S_Z>1$, then $S_Z$ is hyperbolic,
that is, the signature of $S_Z$ is $(1, \rank S_Z-1)$.
For a positive integer $n$ with $n\equiv 2\bmod 8$,
let $L_n$ denote an even unimodular hyperbolic lattice of rank $n$,
which is unique up to isomorphism.
\emph{Borcherds' method}~\cite{Bor1, Bor2} is a procedure to calculate
the automorphism group $\Aut(X)$ of a $K3$ surface $X$
by embedding $S_X$  primitively into 
$L_{26}$,
and applying Conway's result~\cite{Conway1983} on $\OG(L_{26})$.
After the work of Kondo~\cite{Kondo1998},
this method has been applied to many $K3$ surfaces,
and automatized for computer calculation (see~\cite{Shimada2015}  and the references therein).
\par
Let $Y$ be an Enriques surface in characteristic $\ne 2$
with the universal covering $\pi\colon X\to Y$.
Then we have a primitive embedding $\pi^*\colon S_Y(2)\inj S_X$,
where, for a lattice $L$, we denote by $L(2)$ the lattice
with the same underlying $\Z$-module as $L$
and with the symmetric bilinear form being two times that of $L$.
Note that $S_Y$ is isomorphic to $L_{10}$.
Hence, to extend Borcherds' method to
Enriques surfaces,
it is important to study the primitive embeddings of $L_{10}(2)$ into $L_{26}$.
\par
We identify $\OG(L_{10}(2))$ with $\OG(L_{10})$.
We say that two embeddings $\iota$ and $\iota\sprime$ of $L_{10}(2)$ into $L_{26}$ are
\emph{equivalent  up to the action of $\OG(L_{10})$ and $\OG(L_{26})$}
if there exist isometries $g\in \OG(L_{10})$ and $g\sprime \in \OG(L_{26})$
such that, for all $v\in L_{10}(2)$, one has  $\iota(v)^{g\sprime}=\iota\sprime(v^g)$.
Our first main result is as follows.
\begin{theorem}\label{thm:seventeen}
Up to the action of $\OG(L_{10})$ and $\OG(L_{26})$,
there exist exactly $17$ equivalence classes of
primitive embeddings of $L_{10}(2)$ into $L_{26}$,
and they are given in Table~\ref{table:seventeen}.
\end{theorem}
\begin{table}
\[
\renewcommand{\arraystretch}{1.2}
\begin{array}{cclll}
\texttt{No.}&\texttt{name}&\texttt{rt}&\texttt{m4}&\texttt{og}\\ 
\hline 
\mystruth{12pt}
1&\texttt{12A}&D_{8}&1376& {2}^{29} \cdot {3}^{7} \cdot {5}^{3} \cdot {7}^{2} \\ 
2&\texttt{12B}&A_{7}&1824& {2}^{23} \cdot {3}^{6} \cdot {5}^{2} \cdot {7}^{2} \\ 
3&\texttt{20A}&D_{4}+D_{5}&1760& {2}^{25} \cdot {3}^{7} \cdot {5}^{2} \cdot {7} \\ 
4&\texttt{20B}&2D_{4}&1888& {2}^{29} \cdot {3}^{4} \cdot {5} \cdot {7} \\ 
5&\texttt{20C}&10A_{1}+D_{6}&1632& {2}^{28} \cdot {3}^{6} \cdot {5}^{3} \cdot {7} \\ 
6&\texttt{20D}&A_{3}+A_{4}&2016& {2}^{16} \cdot {3}^{6} \cdot {5}^{3} \cdot {7} \\ 
7&\texttt{20E}&5A_{1}+A_{5}&1952& {2}^{20} \cdot {3}^{7} \cdot {5}^{3} \\ 
8&\texttt{20F}&2A_{3}&2080& {2}^{23} \cdot {3}^{4} \cdot {5}^{2} \\ 
9&\texttt{40A}&4A_{1}+2A_{3}&2016& {2}^{25} \cdot {3}^{5} \cdot {5} \\ 
10&\texttt{40B}&8A_{1}+2D_{4}&1760& {2}^{30} \cdot {3}^{6} \cdot {5} \cdot {7} \\ 
11&\texttt{40C}&6A_{1}+A_{3}&2080& {2}^{20} \cdot {3}^{5} \cdot {5} \cdot {7} \\ 
12&\texttt{40D}&12A_{1}+D_{4}&1888& {2}^{28} \cdot {3}^{5} \cdot {5}^{2} \\ 
13&\texttt{40E}&2A_{1}+2A_{2}&2144& {2}^{16} \cdot {3}^{6} \cdot {5}^{2} \\ 
14&\texttt{96A}&8A_{1}&2144& {2}^{28} \cdot {3}^{3} \\ 
15&\texttt{96B}&16A_{1}&2016& {2}^{31} \cdot {3}^{5}\\ 
16&\texttt{96C}&4A_{1}&2208&  {2}^{22} \cdot {3}^{5} \\ 
17&\texttt{infty}& &2272& {2}^{26} \cdot {3}^{2} \cdot {5} \cdot {7} \\ 
\end{array}
\]
\vskip .1cm
\caption{Primitive embeddings}\label{table:seventeen}
\end{table}
\par\noindent
{\bf Explanation of  Table~\ref{table:seventeen}.}
For a lattice $L$, we denote by 
$A(L):=L\dual/L$ 
the \emph{discriminant group} of $L$,
where $L\dual:=\Hom(L, \Z)$ is the dual lattice of $L$,
and by  $\RRR(L)$  the set of $(-2)$-vectors of a lattice $L$.
Let $R_{\iota}$ denote the orthogonal complement of the image of
a primitive embedding $\iota\colon L_{10}(2)\inj L_{26}$ in $L_{26}$.
Note that $R_{\iota}$ is a negative-definite even lattice of rank $16$
with  $A(R_{\iota})$ being isomorphic to $A(L_{10}(2))\cong (\Z/2\Z)^{10}$. 
The item $\texttt{rt}$ is the $\ADE$-type of the negative-definite root lattice
generated by  $\RRR(R_{\iota})$.
For the embedding 
$\texttt{infty}$, the lattice $R_{\iota}$ contains no $(-2)$-vectors.
The item $\texttt{m4}$ is the number of $(-4)$-vectors   in $R_{\iota}$.
The item $\texttt{og}$ is the order of the group  $\OG(R_{\iota})$.
\par
\medskip
These $17$ embeddings have a remarkable property,
which is very useful for the calculation of the automorphism group of
an Enriques surface.
In order to state this property,
we need to explain the notion of \emph{tessellation by chambers}.
Let $L$ be an even hyperbolic lattice.
A \emph{positive cone} $\PPP(L)$  is one of the two connected components
of the subspace of $L\tensor \R$ consisting of vectors $x\in L\tensor \R$
such that $\intf{x, x}>0$.
We fix a positive cone $\PPP(L)$ of $L$, and
denote by $\OGP{L}$
the stabilizer subgroup of $\PPP(L)$ in $\OG(L)$,
which is of index $2$ in $\OG(L)$.
A \emph{rational hyperplane} $(v)\sperp$ is a subspace of $\PPP(L)$ defined by
$\intf{x, v}=0$, where $v\in L\tensor \Q$ is a vector satisfying  $\intf{v, v}<0$.
Let $\FFF$ be a locally finite family of rational hyperplanes of $\PPP(L)$.
A closed subset $D$ of $\PPP(L)$ is said to be an \emph{$\FFF$-chamber} if
$D$ is the closure in $\PPP(L)$ of a connected component of
the complement
\[
\PPP(L)\setminus \bigcup_{H\in \FFF}\; H.
\]
We say that  a subset $N$ of $\PPP(L)$ has a \emph{tessellation by $\FFF$-chambers}
if $N$ is a union of $\FFF$-chambers.
For example, if $\FFF\sprime$ is a subfamily of $\FFF$,
then every $\FFF\sprime$-chamber has a tessellation by $\FFF$-chambers.
\begin{definition}
Note that  $\PPP(L)$ has a tessellation by $\FFF$-chambers.
We say that this tessellation of $\PPP(L)$
is \emph{simple} if there exists a subgroup of $\OGP{L}$
that preserves the family $\FFF$
of hyperplanes (and hence the set of $\FFF$-chambers) and acts on
the set of $\FFF$-chambers transitively.
\end{definition}
\begin{definition}
We say that $\FFF$-chambers $D$ and $D\sprime$ are \emph{isomorphic}
if there exists an isometry
$g\in \OGP{L}$ such that $D^g=D\sprime$.
The automorphism group of an $\FFF$-chamber is defined to be
\[
\OG(L, D):=\set{g\in \OGP{L}}{D^g=D}.
\]
\end{definition}
\begin{definition}
Let $D$ be an $\FFF$-chamber,
and $\closure{D}$ the closure of $D$ in $L\tensor\R$.
We say that $D$ is \emph{quasi-finite} if
$\closure{D}\setminus D$ is contained in a union of at most countably many
half-lines $\R_{\ge 0} v_i \subset \bdr\overline{\PPP}(L)$, where
$v_i$ are non-zero vectors of $L\tensor\R$ satisfying $\intf{v_i, v_i}=0$,
$\overline{\PPP}(L)$ is the closure of $\PPP(L)$ in $L\tensor\R$,
and $\bdr\overline{\PPP}(L):=\overline{\PPP}(L)\setminus \PPP(L)$.
\end{definition}
Each $(-2)$-vector  $r\in \RRR(L)$ defines the \emph{reflection}
$s_r\in \OGP{L}$ into the mirror $(r)\sperp$,
which is defined by $x^{s_r}:= x+\intf{x, r}r$.
Let $W(L)$ denote the subgroup  of $\OGP{L}$
generated by reflections $s_r$, where $r$ runs through  $\RRR(L)$.
\begin{example}\label{example:DRRR}
We put $\RRR(L)\sperp:=\shortset{(r)\sperp}{r\in \RRR(L)}$,
which is a locally finite family of rational hyperplanes.
Then an $\RRR(L)\sperp$-chamber $D_{\RRR}$ is a standard fundamental domain
of the action on $\PPP(L)$ of  $W(L)$.
Hence the tessellation of $\PPP(L)$ by $\RRR(L)\sperp$-chambers
is simple.
Note that we have $\OGP{L}=  W(L)\semidirectproduct \OG(L, D_{\RRR})$.
\end{example}
\begin{definition}
The shape of an $\RRR(L_n)\sperp$-chamber was determined by Vinberg~\cite{Vinberg1975}
for $n=10$ and  $18$, and by Conway~\cite{Conway1983} for $n=26$.
Hence we call  an $\RRR(L_{10})\sperp$-chamber a \emph{Vinberg chamber},
and an $\RRR(L_{26})\sperp$-chamber a \emph{Conway chamber}.
\end{definition}
It is known that Vinberg chambers and Conway chambers are quasi-finite.
\begin{definition}
Let $D$ be an \emph{$\FFF$-chamber}.
A \emph{wall} of $D$ is a closed subset $w$ of $D$
disjoint from the interior of $D$ satisfying the following;  there exists a hyperplane $(v)\sperp \in \FFF$
such that  $w$ is equal to $D\cap (v)\sperp$ and that $w$ contains a non-empty open subset of $(v)\sperp$.
We say that  $v\in L\tensor \Q$ \emph{defines a wall $w$} of $D$ if $w$ is equal to $D\cap (v)\sperp$
and $\intf{x, v}\ge 0$ holds for all $x\in D$.
\end{definition}
\begin{example}
Let $D_{\RRR}$ be as in Example~\ref{example:DRRR}.
Then  the group $W(L)$ is generated by reflections with respect to the $(-2)$-vectors
defining walls of $D_{\RRR}$.
\end{example}
\begin{definition}
Let $D$  be an $\FFF$-chamber, and $w$ a wall of $D$.
Then there exists a unique $\FFF$-chamber $D\sprime$ such that $D\cap D\sprime=w$.
We call $D\sprime$ the $\FFF$-chamber \emph{adjacent to $D$ across the wall $w$}.
\end{definition}
Let $\iota\colon S\inj L$ be an embedding
of an even hyperbolic lattice $S$,
$\PPP(S)$  the positive cone of $S$
that is mapped to $\PPP(L)$ by $\iota\tensor \R$,
and
$\iota_{\PPP}\colon \PPP(S)\inj \PPP(L)$
the induced inclusion.
We put
\[
\iota^{*} \FFF:=\set{\iota_{\PPP}\inv (H)}{H\in \FFF, \;
\emptyset\ne \iota_{\PPP}\inv (H) \subsetneq \PPP(S)\;
}.
\]
Then $\iota^{*} \FFF$ is a locally finite family of rational hyperplanes of $\PPP(S)$,
and $\PPP(S)$ has a tessellation by $\iota^{*} \FFF$-chambers.
If all $\FFF$-chambers are quasi-finite,
then so are all $\iota^{*}\FFF$-chambers.
\par
In the following,
we identify $\PPP(L_{10}(2))$ with $\PPP(L_{10})$.
If $\iota\colon L_{10}(2)\inj L_{26}$ is a primitive embedding,
then $\PPP(L_{10})$ has a tessellation by $\iota^{*} \RRR(L_{26})\sperp$-chambers.
We call an $\iota^{*} \RRR(L_{26})\sperp$-chamber
an \emph{induced chamber} associated with the embedding $\iota$.
Note that every induced chamber is quasi-finite.
\par
In the application of Borcherds' method for the calculation of $\Aut(X)$
of a $K3$ surface $X$,
we embed $S_X$ into $L_{26}$ primitively and investigate
the tessellation of $\PPP_X$ by induced chambers.
This tessellation
is usually not simple,
and in these cases, the computation of $\Aut(X)$ becomes very hard.
See, for example,  the case of the singular $K3$ surface with transcendental lattice
of discriminant $11$ treated in~\cite{Shimada2015},
or the case of the supersingular $K3$ surface of Artin invariant $1$ in characteristic $5$
studied in~\cite{KatsuraKondoShimada2014}.
\par
Our second main result is as follows.
\begin{theorem}\label{thm:simple}
Let $\iota\colon L_{10}(2)\inj L_{26}$ be a  primitive embedding
that is not of type ${\tt infty}$.
Then the  number of walls of an induced chamber $D$ is finite,
and  each wall of  $D$ is defined by a $(-2)$-vector of $L_{10}$.
If $r\in \RRR(L_{10})$ defines a wall $w=D\cap (r)\sperp$ of $D$,
then the reflection $s_r$ with respect to $r$  preserves
the family of hyperplanes $\iota\sp*\RRR(L_{26})\sperp$
and hence the set of induced chambers.
In particular, the induced chamber adjacent to $D$
across the wall $w=D\cap (r)\sperp$
is equal to  $D^{s_r}$.
\end{theorem}
\begin{corollary}\label{cor:simple}
If $\iota\colon L_{10}(2)\inj L_{26}$ is not of type  ${\tt infty}$,
then the tessellation of $\PPP(L_{10})$ by induced chambers is simple.
\end{corollary}
\begin{table}
\[
\renewcommand{\arraystretch}{1.2}
\begin{array}{ccllllll}
\texttt{No.}&\texttt{name}&\texttt{walls}&\texttt{volindex}&\texttt{gD}&\texttt{orb}&\texttt{isom}&\texttt{NK}\\ 
\hline 
\mystruth{12pt}
1&\texttt{12A}&12& {2}^{12} \cdot {3}^{5} \cdot {5}^{2} \cdot {7} & {2}^{2}  &2+2+2+2+4& &\textrm{I}\\ 
2&\texttt{12B}&12& {2}^{12} \cdot {3}^{3} \cdot {5} \cdot {7}& {2}^{3} \cdot {3}  &6+6& &\textrm{II}\\ 
3&\texttt{20A}&20& {2}^{8} \cdot {3}^{4} \cdot {5} \cdot {7}& {2}^{3} \cdot {3}  &4+4+6+6& &\textrm{V}\\ 
4&\texttt{20B}&20& {2}^{10} \cdot {3}^{2} \cdot {5} \cdot {7}& {2}^{6} &4+8+8& &\textrm{III}\\ 
5&\texttt{20C}&20& {2}^{6} \cdot {3}^{3} \cdot {5} \cdot {7} & {2}^{3} \cdot {3} \cdot {5}  &5+15&\texttt{20D}&\textrm{VII}\\ 
6&\texttt{20D}&20&{2}^{6} \cdot {3}^{3} \cdot {5} \cdot {7} & {2}^{3} \cdot {3} \cdot {5} &5+15&\texttt{20C}&\textrm{VII}\\ 
7&\texttt{20E}&20& {2}^{7} \cdot {3}^{4} \cdot {5} & {2}^{3} \cdot {3} \cdot {5} &10+10& &\textrm{VI}\\ 
8&\texttt{20F}&20& {2}^{9} \cdot {3}^{2} \cdot {5} & {2}^{6} \cdot {5}  &20& &\textrm{IV}\\ 
9&\texttt{40A}&40& {2}^{7} \cdot {3}^{2} \cdot {5}& {2}^{7} \cdot {3} &12+12+16& & \\ 
10&\texttt{40B}&40& {2}^{3} \cdot {3}^{2} \cdot {5} \cdot {7}& {2}^{7} \cdot {3}^{2}  &16+24&\texttt{40C}& \\ 
11&\texttt{40C}&40& {2}^{3} \cdot {3}^{2} \cdot {5} \cdot {7}& {2}^{7} \cdot {3}^{2}  &16+24&\texttt{40B}& \\ 
12&\texttt{40D}&40& {2}^{5} \cdot {3}^{2} \cdot {5}& {2}^{5} \cdot {3}^{2} \cdot {5}  &10+30&\texttt{40E}& \\ 
13&\texttt{40E}&40&{2}^{5} \cdot {3}^{2} \cdot {5} & {2}^{5} \cdot {3}^{2} \cdot {5}  &10+30&\texttt{40D}& \\ 
14&\texttt{96A}&96&{2}^{5} \cdot {3}^{2} & {2}^{13} \cdot {3}  &32+64& & \\ 
15&\texttt{96B}&96&{2}^{3} \cdot {3}^{2}& {2}^{12} \cdot {3}^{3} &96&\texttt{96C}& \\ 
16&\texttt{96C}&96&{2}^{3} \cdot {3}^{2}& {2}^{12} \cdot {3}^{3}  &96&\texttt{96B}& \\ 
17&\texttt{infty}&\infty& & & & & \\ 
\end{array}
\]
\vskip .1cm
\caption{Induced chambers}\label{table:inducedchams}
\end{table}
\par
The data of the induced chambers $D$ are given in Table~\ref{table:inducedchams}.
Before explaining the contents of Table~\ref{table:inducedchams},
we recall two classical results about automorphism groups of  Enriques surfaces.
Let $Y$ be an Enriques surface.
We denote by  $\PPP_Y$  the positive cone of $S_Y\tensor \R$
containing an ample class.
We then put
\[
N_{Y}:=\set{x\in \PPP_Y}{\intf{x, [\Gamma]}\ge 0\;\;\textrm{for all curves $\Gamma$ on $Y$}\;}.
\]
Then $N_{Y}$ has a tessellation by Vinberg chambers,
because $N_{Y}$ is bounded by the hyperplanes $([\Gamma])\sperp$
defined by the classes $[\Gamma]$ of smooth rational curves $\Gamma$ on $Y$
and every smooth rational curve on $Y$ has the self-intersection number $-2$.
\par
Let $Y$ be a complex \emph{generic} Enriques surface.
Then we have $\PPP_Y=N_Y$.
Barth and Peters~\cite{BP1983} showed that
$\Aut(Y)$ is canonically identified with the kernel of
the $\bmod\, 2$-reduction homomorphism
$\OGP{S_Y}\to \OGP{S_Y}\tensor \F_2$.
Since a Vinberg chamber has no automorphism group,
the group  $\OGP{S_Y}$ is equal to the subgroup
$W(S_Y)$.
Since the $\bmod\, 2$-reduction homomorphism above is
surjective~(see~\cite{BP1983} and Section~\ref{subsec:proof17} of this paper),
there exists a union $\VVV$ of
\[
| \OGP{S_Y}\tensor \F_2|=46998591897600=\oLtenFtwo
\]
Vinberg chambers such that (i) $\PPP(L_{10})$ is the union of $\VVV^g$,
where $g$ runs through $\Aut(Y)$,
and (ii) if $g\in \Aut(Y)$ is not the identity,
then the interiors of $\VVV$ and of $\VVV^g$ are disjoint.
\par
Kondo~\cite{Kondo1986} and Nikulin~\cite{Nikulin1984}
classified all complex Enriques surfaces with finite automorphism group.
This classification was extended to odd characteristics by Martin~\cite{Martin2019}.
It turns out that Enriques surfaces in characteristic $\ne 2$
with finite automorphism group are divided into $7$ classes
\textrm{I}, \dots, \textrm{VII}.
An Enriques surface $Y$ with finite automorphism group
has only a finite number of smooth rational curves $\Gamma$,
and $N_{Y}$ is bounded by the hyperplanes
$([\Gamma])\sperp$ defined by these  curves.
The configurations of these smooth rational curves
are explicitly depicted in~\cite{Kondo1986}.
\par
\medskip
{\bf Explanation of  Table~\ref{table:inducedchams}.}
The item \texttt{walls} is the number of walls of an induced chamber $D$.
Since every wall of $D$ is defined by a $(-2)$-vector of $L_{10}$,
it follows that $D$ is a union of Vinberg chambers.
The item \texttt{volindex} shows that
the number of Vinberg chambers contained in $D$ is equal to
\[
|  \OGP{S_Y}\tensor \F_2|/\texttt{volindex}=\oLtenFtwo/\texttt{volindex}.
\]
The item \texttt{gD} is the order of the automorphism group $\OG(L_{10}, D)$ of $D$.
The item \texttt{orb} describes the orbit decomposition of the set of walls
under the action of $\OG(L_{10}, D)$.
The item \texttt{isom} shows that, for example,
the induced chambers of  the primitive embeddings \texttt{20C}
and \texttt{20D} are isomorphic.
The item \texttt{NK} shows that, for example,
the induced chamber of  the primitive embedding $\texttt{12A}$
is, under a suitable isomorphism $L_{10}\cong S_Y$,  equal to
 $N_{Y}$ of an Enriques surface $Y$ 
with finite automorphism group 
of type $\textrm{I}$.
\par
\medskip
Since all $7$ types \textrm{I}, \dots, \textrm{VII} appear in the column \texttt{NK},
our results on the induced chamber $D$ can be applied to $N_Y$
for an arbitrary Enriques surface $Y$ with finite automorphism group
in characteristic $\ne 2$.
\par
Borcherds' method has been applied to Enriques surfaces
in~\cite{Shimada2016}~and~\cite{ShimadaHessian}
without using the facts proved in this paper.
These facts actually give us a big advantage in
the calculation of the automorphism group $\Aut(Y)$
of an Enriques surface $Y$ by Borcherds' method,
as is exemplified 
in~\cite{BRS2019}~and~\cite{ShimadaVeniani2019}.
We can also enumerate all polarizations of $Y$ with a fixed degree
modulo $\Aut(Y)$ by means of the method in~\cite{ShimadaHoles}.
These applications will be treated in other papers.
In~\cite{DKbook}, 
an interesting relation of 
our list with Enriques surfaces in characteristic $2$ 
is discussed.
%
\par
For the computation, the first author used a mixture of {\tt SageMath, PARI, GAP} \cite{sagemath,PARI2,GAP},
and the second author used {\tt GAP}~\cite{GAP}.
The explicit computational data is available
at the second author's webpage~\cite{thecompdataL10L26}.
\par
Thanks are due to Professor Igor Dolgachev and Professor Shigeyuki Kondo
for their interests in this work and many comments.
The authors also thank the referee for his/her valuable comments.
\par
\medskip
{\bf Notation.}
To avoid possible confusions between $L_{10}$ and $L_{10}(2)$,
we put
\[
\SS:=L_{10}(2).
\]
We identify the underlying $\Z$-modules of $L_{10}$ and $\SS$, and
choose positive cones so that $\PPP(L_{10})=\PPP(\SS)$.
We also have a natural identification $\OGP{L_{10}}=\OGP{\SS}$.
We denote by $\intfS{\,,\,}$,  $\intfL{\,,\,}$ and $\intfLL{\,,\,}$
the symmetric bilinear forms of $\SS$,   $L_{10}$, and $L_{26}$, respectively.
\section{Proof of Theorems~\ref{thm:seventeen}~and~\ref{thm:simple}}\label{sec:proof}
\subsection{Discriminant form}\label{subsec:discform}
Let $L$ be an even lattice.
Recall that $A(L)=L\dual/L$ is the discriminant group of $L$.
The  quadratic form
\[
q(L)\colon A(L)\to \Q/2\Z
\]
defined by $u \bmod L \mapsto \intf{u,u}\bmod 2\Z$ for $u\in L\dual$
is called the \emph{discriminant form} of $L$.
Let $\OG(q(L))$ denote the automorphism group of the finite quadratic form $q(L)$.
Then we have a natural homomorphism
\[
\eta(L)\colon \OG(L)\to \OG(q(L)).
\]
See Nikulin~\cite{Nikulin1979} for the basic properties of discriminant forms.
Among these properties, the following is especially important for us:
\begin{proposition}\label{prop:overlat}
Let $M$ and $N$ be even lattices.
We consider the following sets:
\begin{enumerate}[\rm (a)]
\item the set $\LLL$ of even unimodular lattices $L$ contained in  $M\dual\oplus N\dual$,
containing $M\oplus N$,
and containing each of $M$ and $N$  primitively, and
\item the set $\QQQ$ of isomorphisms between
the finite quadratic forms $q(M)$ and $-q(N)$.
\end{enumerate}
Let $\phi$ be an isomorphism from $q(M)$ to $-q(N)$,
let $\varGamma_{\phi}\subset A(M)\oplus A(N)$ denote the graph of $\phi$,
and let $L_{\phi}\subset M\dual\oplus N\dual$ be the pull-back of $\varGamma_{\phi}$
by the natural projection $M\dual\oplus N\dual\to A(M)\oplus A(N)$.
Then the mapping  $\phi\mapsto L_{\phi}$ gives rise to a bijection from $\QQQ$ to $\LLL$.
This  bijection $\QQQ\cong \LLL$
is compatible with the natural actions of $\OG(M)\times \OG(N)$
on $\QQQ$ and on $\LLL$.
\qed
\end{proposition}
Suppose that $L\in \LLL$,
so that $N$ is the orthogonal complement of the primitive sublattice $M\subset L$.
Let $\phi\colon q(M) \isom -q(N) $ be the isomorphism  corresponding to $L$,
and $\OG(\phi)\colon \OG(q(M)) \isom \OG(q(N))$ the induced isomorphism.
We put
\[
\OG(L, M):=\set{\tilde{g}\in \OG(L)}{\,\textrm{$\tilde{g}$ preserves $M$}\,},
\]
and let $\tilde{g}\mapsto \tilde{g}|M$ and $\tilde{g}\mapsto \tilde{g}|N$
denote the restriction homomorphisms from $\OG(L, M)$ to $\OG(M)$ and $\OG(N)$, respectively.
We say that $\tilde{g}\in \OG(L, M)$ is a \emph{lift} of $g\in \OG(M)$
if $\tilde{g}|M=g$.
\begin{corollary}\label{cor:lifts}
Let $g$ be an isometry of $M$.
Then the  homomorphism
$\tilde{g}\mapsto \tilde{g}|N$ induces a bijection
from the set of lifts $\tilde{g}$ of $g$ to the set of all isometries $h\in \OG(N)$ of $N$
such that $\eta(M)(g)\in \OG(q(M))$ is mapped to $\eta(N)(h)\in \OG(q(N))$ by $\OG(\phi)$.
\qed
\end{corollary}
\subsection{Kneser's neighbor method}
This method allows us to efficiently compute all lattices in a given genus.
We review the basic idea.
For proofs and a more complete treatment, see \cite{kneser} and  \cite{neighbor}.
In this subsection, we assume that 
all lattices are positive or negative definite.
\par
Recall that two lattices $L$ and $L'$ are \emph{in the same genus} if
we have isomorphisms
\[
L\otimes \Z_p \cong L'\otimes \Z_p \;\;\mbox{  and  }\;\; L \otimes \R \cong L' \otimes \R
\]
of $\Z_p$- or $\R$-valued quadratic modules for every prime $p$,
where $\Z_p$ denotes the ring of  $p$-adic integers.
Suppose that $L$ and $L'$ are in the same genus.
Then,  by the Hasse--Minkowski theorem, we have  $L \otimes \Q\cong L' \otimes \Q$.
Thus we may and will assume that $L \otimes \Q = L' \otimes \Q$.
Let $p$ be an odd prime  which does not divide the determinant $\det L:=|A(L)|$ of $L$.
We say that two lattices $L$ and $L'$ are \emph{$p$-neighbors} if
\[
p=[L: L \cap L']=[L': L \cap L'].
\]
Suppose that  $L$ and $L'$ are $p$-neighbors.
Then $L \otimes \Z_q = L' \otimes \Z_q$ for all primes $q \neq p$.
Moreover,
since $p$ does not divide $\det L$,
both $L \otimes \Z_p$ and $L' \otimes \Z_p$ are unimodular $\Z_p$-lattices isomorphic over
the field of $p$-adic rationals $\Q_p$.
Thus $L\otimes \Z_p$ and $L' \otimes \Z_p$ are in fact isomorphic.
We have proved that the $p$-neighbors $L$ and $L'$ are in the same genus.
\par
For a given genus $\mathcal{G}$, we denote by $C$ the set of isomorphism classes $[L]$
of lattices $L$ in this genus.
Let $p$ be an odd prime.
Set
\[
E := \{([L],[L']) \in C \times C \mid L \mbox{ and } L' \mbox{ are } p \mbox{-neighbors} \}.
\]
Then $(C,E)$ is called the \emph{$p$-neighbor graph} of $\mathcal{G}$.
Assume further that $L \otimes \Z_p$ represents $0$
for a lattice $L$ in this genus.
This is certainly the case
if the rank of $L$ is at least $5$.
In general each connected component of this graph is
the union of several so called proper spinor genera.
In the case relevant to us,
the genus consists of a single proper spinor genus, so this does not concern us.
\par
For given $L$ and $v \in L \setminus pL$ with $\intf{v,v} \in p^2 \Z_p$,
the lattice
\[
L(v):=L_v + \Z (v/p)\mbox{ where }L_v = \{x \in L \mid \intf{x,v} \in p\Z\}
\]
is called
the \emph{$p$-neighbor of $L$ with respect to $v$}.
One can show that
$L$ and $L(v)$ are indeed $p$-neighbors, that $L_v$ depends only on $v \mod pL$
(as long as $\intf{v,v}$ stays divisible by $p^2$),
and that every $p$-neighbor of $L$ arises in this fashion.
\par
Thus one can classify lattices in the genus $\mathcal{G}$
by iteratively computing the neighbors of the lattices in $C$ and testing for isomorphism (see \cite{isometries}).
One can speed this up by computing the neighbors of a given lattice only up to the action of the orthogonal group.
When we are interested only in the vertices and not in the edges,
we can break the computation when we have ``explored" all vertices.
The \emph{mass} of the genus $\mathcal{G}$ is defined as
\[
\mbox{mass}(\mathcal{G}):= \sum_{[L] \in \mathcal{G}} \frac{1}{|\OG(L)|}.
\]
It can be calculated from the invariants of $\mathcal{G}$ alone as described in \cite{CS1988}.
We can break the computation as soon as
the sum of the reciprocals of $|\OG(L)|$ reaches $\mbox{mass}(\mathcal{G})$.

This procedure is implemented for example in {\tt Magma} \cite{magma}.
In the example relevant to us,
the computation with {\tt Magma} simply exhausted all memory available.
Thus we had to resort to a modified strategy: A random walk through the neighbor graph.
\subsection{Proof of Theorem~\ref{thm:seventeen}}\label{subsec:proof17}
Let $e_1, \dots, e_{10}$ be a basis of $L_{10}$ consisting of $(-2)$-vectors
that form the configuration in Figure~\ref{fig:L10basis}.
Then
\begin{equation}\label{eq:V}
V:=\set{x\in \PPP(L_{10})}{\intfL{x, e_i}\ge 0\;\;\textrm{for}\;\; i=1, \dots, 10}
\end{equation}
is a Vinberg chamber,
and each $V\cap (e_i)\sperp$ is a wall of $V$  (see Vinberg~\cite{Vinberg1975}).
\begin{figure}
\def\ha{40}
\def\hav{37}
\def\hd{25}
\def\hdv{22}
\def\he{10}
\def\hev{7}
\setlength{\unitlength}{1.5mm}
{\small
\begin{picture}(80,11)(-5, 6)
\put(22, 16){\circle{1}}
\put(23.5, 15.5){$e\sb 1$}
\put(22, 10.5){\line(0,1){5}}
\put(9.5, \hev){$e\sb 2$}
\put(15.5, \hev){$e\sb 3$}
\put(21.5, \hev){$e\sb 4$}
\put(27.5, \hev){$e\sb 5$}
\put(33.5, \hev){$e\sb 6$}
\put(39.5, \hev){$e\sb 7$}
\put(45.5, \hev){$e\sb {8}$}
\put(51.5, \hev){$e\sb {9}$}
\put(57.5, \hev){$e\sb {10}$}
\put(10, \he){\circle{1}}
\put(16, \he){\circle{1}}
\put(22, \he){\circle{1}}
\put(28, \he){\circle{1}}
\put(34, \he){\circle{1}}
\put(40, \he){\circle{1}}
\put(46, \he){\circle{1}}
\put(52, \he){\circle{1}}
\put(58, \he){\circle{1}}
\put(10.5, \he){\line(5, 0){5}}
\put(16.5, \he){\line(5, 0){5}}
\put(22.5, \he){\line(5, 0){5}}
\put(28.5, \he){\line(5, 0){5}}
\put(34.5, \he){\line(5, 0){5}}
\put(40.5, \he){\line(5, 0){5}}
\put(46.5, \he){\line(5, 0){5}}
\put(52.5, \he){\line(5, 0){5}}
\end{picture}
}
\caption{Basis of $L_{10}$}\label{fig:L10basis}
\end{figure}
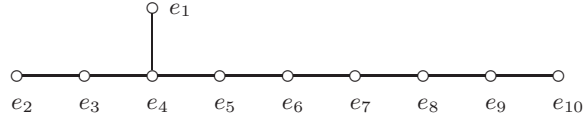
Since the graph in Figure~\ref{fig:L10basis} has no non-trivial automorphisms,
the group $\OGP{L_{10}}$ is generated by the $10$ reflections $s_1, \dots, s_{10}$ with respect to
$e_1, \dots, e_{10}$.
Recall from the paragraph Notation at the end of Introduction that
we put $\SS:=L_{10}(2)$.
In $L_{10}\tensor\Q=\SS\tensor \Q$,
we have
$L_{10}=L_{10}\dual=\SS\subset \SS\dual$,
and the mapping $v\mapsto v/2$ gives an isomorphism $L_{10}\cong \SS\dual$
of $\Z$-modules,
which  gives rise to an isomorphism
\[
(L_{10}/2 L_{10}, q_L)\cong q(\SS),
\quad
\textrm{where $q_{L} (u\bmod 2 L_{10}):=\frac{1}{2}\intfL{u, u}\bmod 2\Z$}
\]
of finite quadratic forms.
Hence we see that $|\OG(q(\SS))|=\oLtenFtwo$
 by Proposition 1.7 of~\cite{BP1983}.
Since we have explicit generators $s_1, \dots, s_{10}$ of $\OGP{L_{10}}=\OGP{\SS}$,
we can confirm that
$\eta(\SS)\colon \OG(\SS)\to \OG(q(\SS))$  restricted to $\OGP{L_{10}}$
is surjective.
%
%
Let $\iota\colon \SS\inj L_{26}$ be a primitive embedding,
and let $R_{\iota}$ be the orthogonal complement of the image of $\iota$ in $L_{26}$.
Then $R_{\iota}$ is of signature $(0, 16)$.
By Proposition~\ref{prop:overlat},
the discriminant form $q(R_{\iota})$ is isomorphic to $-q(\SS)$.
Since $\eta(\SS)$ is surjective,
Proposition~\ref{prop:overlat} implies that,
if a primitive embedding $\iota\sprime\colon \SS\inj L_{26}$
satisfies $R_{\iota\sprime}\cong R_{\iota}$, then
$\iota\sprime$ is equivalent  to $\iota$
up to the action of $\OG(\SS)=\OG(L_{10})$ and $\OG(L_{26})$.
Hence the proof of Theorem~\ref{thm:seventeen} is reduced to the classification of
isomorphism classes of even lattices $R$
with signature $(0, 16)$ such that $q(R)\cong -q(\SS)$.
Note that these conditions on signature and discriminant form
determine the genus $\GGG_R$ of $R$.
By the mass formula~\cite{CS1988},
we see that  the mass of this genus is
\begin{equation*}\label{eq:mass}
 \mathrm{mass}(\GGG_R)=64150367/28766348771328000.
\end{equation*}
\par
Let $u\colon  \F_2^2 \rightarrow \Q / 2\Z$ be defined by $u(x,y):=xy$.
Then one calculates that
\[
-q(\SS)\cong -u^{\oplus 5}=u^{\oplus 5}\cong q(R)
\]
and that $q(D_8)\cong u$ and $q(E_8(2))\cong u^{\oplus 4}$,
where $D_8$ and $E_8$ are the negative-definite root lattices of
$\ADE$-type $D_8$ and $E_8$, respectively.
Thus we have found a first lattice $L=D_8 \oplus E_8(2) $ in $\GGG_R$.
To find representatives up to isomorphism,
we use a variant of Kneser's neighbor method for $p=3$.
Start by inserting $L$ into a list $C$.
Then enter the following loop.
Pick a random $L$ in $C$ and a random $v\in L\setminus 3L$ with $\intf{v,v}$ divisible by $3$,
replace $v$ by $v +3w$ for $w \in L$ such that $\intf{v,v}$ is divisible by $9$.
Calculate the $3$-neighbor $L(v)$ and check if it is isomorphic to any lattice in the list $C$.
If not add it to $C$.
Break the loop when the mass of the lattices in $C$ matches $\mathrm{mass}(\GGG_R)$.
By this computation,
it turns out that $\GGG_R$ is constituted by $17$ isomorphism classes in Table~\ref{table:seventeen},
and hence Theorems~\ref{thm:seventeen} follows.
\qed
\subsection{Conway theory}\label{subsec:weyl}
Let $\weyl$ be a non-zero primitive vector of $L_{26}$
contained in  $\bdr\overline{\PPP}(L_{26})$.
Note that $\intfLL{\weyl, \weyl}=0$.
We put
\[
[\weyl]:=\Z \weyl, \quad [\weyl]\sperp:=\set{v\in L_{26}}{\intfLL{v, \weyl}=0}.
\]
Then  $[\weyl]\sperp/[\weyl]$ has a natural structure of an even unimodular negative-definite lattice,
and hence is isomorphic to $N(-1)$,
where $N$ is one of the $24$ Niemeier lattices
(see, for example, Chapter 16 of~\cite{CSBook}).
\begin{definition}
We say that $\weyl$ is a Weyl vector if $[\weyl]\sperp/[\weyl]$
is isomorphic to the negative-definite Leech lattice.
\end{definition}
Since the Leech lattice is characterized as the unique Niemeier lattice
with no roots,
we can determine whether $\weyl$ is a Weyl vector or not by
calculating the set $\RRR([\weyl]\sperp/[\weyl])$
of $(-2)$-vectors in $[\weyl]\sperp/[\weyl]$.
\par
For a Weyl vector $\weyl$, we put
\[
C(\weyl):=\set{x\in \PPP(L_{26})}{\intfLL{x, r}\ge 0\;\textrm{for all}\; r\in \RRR(L_{26})\;\textrm{with}\; \intfLL{r, \weyl}=1}.
\]
The following theorem is very important.
\begin{theorem}[Conway~\cite{Conway1983}]
The mapping $\weyl\mapsto C(\weyl)$ gives a bijection from the set of Weyl vectors
to the set of Conway chambers.
\qed
\end{theorem}
\begin{remark}\label{rem:weylample}
Let $\weyl$ be a Weyl vector.
Since $\weyl$ is primitive and $L_{26}$ is unimodular,
there exists a vector $\weyl\sprime$ such that
$\intfLL{\weyl\sprime,\weyl\sprime}=0$ and $\intfLL{\weyl,\weyl\sprime}=1$.
Then  every $(-2)$-vector $r$ of $L_{26}$ with $\intfLL{\weyl,r}=1$ is written as
\[
 \alpha_{\lambda}\weyl+\weyl\sprime+\lambda, \;\;\;
 \textrm{where $\alpha_\lambda=\frac{-\intfLL{\lambda, \lambda}-2}{2}$ and $\intfLL{\weyl, \lambda}=\intfLL{\weyl\sprime,\lambda}=0$}.
\]
Since $\intfLL{\weyl, \lambda}=\intfLL{\weyl\sprime, \lambda}=0$ implies  $\intfLL{\lambda, \lambda}\le 0$,
we see that  $a_{26}:=2\weyl+\weyl\sprime$ is an interior point of $C(\weyl)$.
 \end{remark}
\subsection{Proof of Theorem~\ref{thm:simple}}\label{subsec:proofsimple}
In Section~\ref{subsec:proof17}, we have calculated the $17$ primitive embeddings
$\iota\colon \SS\inj L_{26}$ explicitly.
As was said in Notation, we identify $\PPP(\SS)$ and $\PPP(L_{10})$, and
denote by
$\iota_{\PPP}\colon \PPP(L_{10})\inj \PPP(L_{26})$
the induced inclusion.
\par
Our first task is to find a Weyl vector $\weyl$
such that $\iota_{\PPP}\inv (C(\weyl))$ is an induced chamber,
that is, $\iota_{\PPP}\inv  (C(\weyl))$ contains a non-empty open subset of $\PPP(L_{10})$.
Recall that we have fixed a basis $e_1, \dots, e_{10}$ of $L_{10}$.
We put
$a_{10}:=e_1\dual+\cdots+e_{10}\dual$,
where $e_1\dual, \dots, e_{10}\dual$
are the basis of $L_{10}\dual=L_{10}$ dual to $e_1, \dots, e_{10}$.
Then $a_{10}$ is an interior point of the Vinberg chamber $V$ defined by~\eqref{eq:V},
and we have $\intfL{a_{10}, a_{10}}=1240$.
By direct calculation, we confirm the equality
\begin{equation}\label{eq:a10interior}
\shortset{r\in \RRR(L_{26})}{ \intfLL{r, \iota(a_{10})}=0}
\;=\;
\shortset{r\in \RRR(L_{26})}{ \intfLL{r, \iota(x)}=0 \;\textrm{for all}\; x\in L_{10}},
\end{equation}
which means that, if  a hyperplane $(r)\sperp$ of $\PPP(L_{26})$ defined by $r\in \RRR(L_{26})$
passes through $\iota(a_{10})$, then $(r)\sperp$ contains the image of $\iota_{\PPP}$.
(Note that the second set in~\eqref{eq:a10interior} is identified with $\RRR(R_{\iota})$ by
the embedding $R_{\iota}\inj L_{26}$.)
Therefore $a_{10}$ is an interior point of an induced chamber $D$.
\begin{definition}
Let $L$ be an even hyperbolic lattice.
Suppose that $v_1, v_2$ are vectors of $\PPP(L)\cap (L\tensor \Q)$.
We say that a $(-2)$-vector $r\in \RRR(L)$
\emph{separates} $v_1$ and $v_2$ if $\intf{r, v_1}\cdot \intf{r, v_2}<0$.
We can calculate the set
of $(-2)$-vectors of $L$ separating $v_1$ and $v_2$
by the algorithm given in Section 3.3 of~\cite{ShimadaChar5}.
\end{definition}
We perturb $a_{10}$ to $a_{10}\sprime\in \PPP(L_{10})\cap (L_{10}\tensor\Q)$
in a general direction
so that $a_{10}\sprime$ is also an interior
point of the same induced chamber $D$ as  $a_{10}$,
that is, the equality~\eqref{eq:a10interior} remains true with
$\iota(a_{10})$ replaced by $\iota(a_{10}\sprime)$
and  there exist no $(-2)$-vectors $r$ of $L_{26}$
separating $\iota(a_{10})$ and $\iota(a_{10}\sprime)$.
We  choose an arbitrary  Weyl vector $\weyl_0$ of $L_{26}$,
and calculate a vector $a_{26}\in L_{26}$
in the interior of $C(\weyl_0)$
by~Remark~\ref{rem:weylample}.
We then calculate the set $\{\pm r_1, \dots, \pm r_N\}$
of $(-2)$-vectors of $L_{26}$
separating $\iota(a_{10}\sprime)$ and $a_{26}$.
We sort these $(-2)$-vectors  $r_1, \dots, r_N$
in such a way that the line segment
from $a_{26}$ to $\iota(a_{10}\sprime)$
intersects the hyperplanes $(r_1)\sperp, \dots, (r_N)\sperp$ in this order.
Since $a_{10}\sprime$ is a result of  general perturbation,
these $N$ intersection points
are distinct.
Let $s_{\nu}\in \OGP{L_{26}}$ be the reflection with respect to $r_{\nu}$.
We  move $\weyl_0$   by $s_1, \dots, s_N$ in this order and obtain a new Weyl vector
$\weyl:=\weyl_0^{s_1\cdots s_N}$.
Then $C(\weyl)$ contains $a_{26}^{s_1\cdots s_N}$ in its interior,
and there exist no $(-2)$-vectors of $L_{26}$ separating
$\iota(a_{10}\sprime)$ and $a_{26}^{s_1\cdots s_N}$.
Therefore $\iota_{\PPP}\inv (C(\weyl))$ is the induced chamber $D$ containing $a_{10}$
in its interior.
\par
Next we calculate the set of walls of the induced chamber $D=\iota_{\PPP}\inv  (C(\weyl))$.
We denote by $v\mapsto v_{\SS}$ and $v\mapsto v_R$
the orthogonal projections $L_{26}\to \SS\dual$ and $L_{26}\to R_{\iota}\dual$,
and let $\intfR{\,,\,}$ denote the symmetric bilinear form  of $R_{\iota}$.
It turns out that  $\intfS{\weyl_{\SS}, \weyl_{\SS}}>0$ holds
except for the case where $\iota$ is of type  ${\tt infty}$.
Henceforth we assume that $\iota$ is not of type  ${\tt infty}$.
We put
\begin{eqnarray*}
\RRR(L_{26}, \weyl) &:=&\set{r\in \RRR(L_{26})}{\intfLL{\weyl, r}=1}, \\
\RRR(L_{26}, D) &:=&\set{r\in  \RRR(L_{26}, \weyl)}{\intfS{r_{\SS}, r_{\SS}}<0}, \\
\RRR_D&:=&\set{r_{\SS}}{ r\in \RRR(L_{26}, D)}.
\end{eqnarray*}
For $r\in \RRR(L_{26})$,
the hyperplane $(r)\sperp$ of $\PPP(L_{26})$
intersects  the image of $\iota_{\PPP}$
if and only if $\intfS{r_{\SS}, r_{\SS}}<0$.
Hence we have
\[
D=\set{x\in \PPP(L_{10})}{\intfS{r_{\SS}, x}\ge 0\;\;\textrm{for all}\;\; r_{\SS}\in \RRR_D}.
\]
The set $\RRR_D$ can be calculated explicitly as follows.
Suppose that
 $r\in \RRR(L_{26}, D)$.
Then we have
\[
\intfS{\weyl_{\SS}, r_{\SS}}+\intfR{\weyl_R, r_R}=1,
\quad
\intfS{r_{\SS}, r_{\SS}}+\intfR{r_R, r_R}=-2.
\]
Since $R_{\iota}$ is negative-definite,
the condition $\intfS{r_{\SS}, r_{\SS}}<0$ implies  $-2<\intfR{r_R, r_R}\le 0$,
and if $\intfR{r_R, r_R}=0$,
then we would have $r_R=0$, $r=\iota(r_{\SS})$ and hence $r_{\SS}\in \SS$,
which is impossible because $\intfLL{r,r}=-2$ whereas
$\intfS{r_{\SS}, r_{\SS}}=2\intfL{r_{\SS}, r_{\SS}}$ is a multiple of $4$.
Since $\SS\dual = \frac{1}{2} \SS$,
the discriminant form of $\SS$ takes values in $\Z / 2\Z$.
Naturally, the same is true for $q(R_{\iota})\cong -q(\SS)$.
This means that $\intfR{v,v}$ is integral for any $v \in R_{\iota}\dual$.
In particular, if $v\in R_{\iota}\sp{\vee}$ satisfies $-2<\intfR{v, v}<0$, then $\intfR{v, v}=-1$.
Therefore we  have $r_R\in V_R$ for all $r\in \RRR(L_{26}, D)$, where
\[
V_R:=\set{v\in R_{\iota}\dual}{\intfR{v, v}=-1}.
\]
For $v\in V_R$, we put $a(v):=1-\intfR{\weyl_R, v}$,
and let $a(V_R)$ be the set $\shortset{a(v)}{v\in V_R}$.
For each $a\in a(V_R)$, we calculate
\[
V_{\SS}(a):=\set{u\in \SS\dual}{\intfS{\weyl_{\SS}, u}=a, \; \intfS{u, u}=-1},
\]
which is finite because $\intfS{\weyl_{\SS}, \weyl_{\SS}}>0$.
We calculate the sum $u+v \in \SS\dual \oplus R_{\iota}\dual$ for all pairs $(v, u)$
of $v\in V_R$ and $u\in V_{\SS}(a(v))$,
and check whether $u+v$ is in $L_{26}$ or not.
Thus we calculate
\[
\RRR(L_{26}, D)=\; L_{26}\;\cap \; \set{u+v}{v\in V_R, \;u\in V_{\SS}(a(v))}
\]
and the set $\RRR_D$.
It turns out that
the hyperplanes $(r_{\SS})\sperp$ for $r_{\SS}\in \RRR_D$ are distinct,
and that $\RRR_D$ spans $L_{10}\tensor\Q$ over $\Q$.
We will show that each $r_{\SS}\in \RRR_D$ defines a wall of $D$.
We can calculate the finite group $G$ of isometries of $L_{10}$
that preserves the finite set $\RRR_D\subset L_{10}$.
Recall that $a_{10}$ is an interior point of $D$.
We put
$\sigma_{10}:=\sum_{g\in G} a_{10}^g$,
which is
an interior point  of $D$
fixed by $G$.
We put $t:=-\intfS{\sigma_{10}, r_{\SS}}/\intfS{r_{\SS}, r_{\SS}}$,
and consider the point $\sigma\sprime_{10}:=\sigma_{10}+ t\cdot r_{\SS}$  on $(r_{\SS})\sperp$.
Then we have $\intfS{\sigma\sprime_{10}, r\sprime_{\SS}}>0$ for all
$r\sprime_{\SS} \in \RRR_D\setminus \{r_{\SS}\}$,
which means that $\sigma\sprime_{10}$ is an interior point of
the subset $D\cap (r_{\SS})\sperp$ of $(r_{\SS})\sperp$.
Therefore  $D\cap (r_{\SS})\sperp$  is a wall of $D$.
Note that, since $\intfS{r_{\SS}, r_{\SS}}=-1$ for $r_{\SS}\in \RRR_D$
and $2\SS\dual \subset \SS$, we see that $2r_{\SS}\in L_{10}$ and $\intfL{2 r_{\SS}, 2 r_{\SS}}=-2$.
Therefore each wall of $D$ is defined by a $(-2)$-vector $2r_{\SS}$ of $L_{10}$.
The group $G$ is equal to $\OG(L_{10}, D)$.
\par
The assertions in Theorem~\ref{thm:simple} and Table~\ref{table:inducedchams}
about the walls of an induced chamber are now proved
for the induced chamber $D=\iota_{\PPP}\inv(C(\weyl))$
defined by this particular Weyl vector $\weyl$.
The data \texttt{volindex} of $D$ in Table~\ref{table:inducedchams}
is calculated by the method given in~\cite{ShimadaHessian}.
\par
To prove that
the induced tessellation of $\PPP(L_{10})$ is simple and thus complete the proof of Theorem~\ref{thm:simple},
it is enough to prove the following:
\begin{proposition}\label{prop:gtilde}
For each wall $D\cap (r)\sperp$ of $D$, where $r\in \RRR(L_{10})$,
there exists an isometry $\tilde{g}\in \OGP{L_{26}}$
with the following property:
the isometry $\tilde{g}$
preserves the image of $\iota\colon\SS\inj L_{26}$
and its restriction  $\tilde{g}|\SS\in \OGP{\SS}=\OGP{L_{10}}$ to $\SS$
is equal to  the reflection $s_r\in \OGP{L_{10}}$  with respect to
$r\in \RRR(L_{10})$.
\end{proposition}
Suppose that Proposition~\ref{prop:gtilde} is proved.
Since $\tilde{g}$ preserves $\RRR(L_{26})$,
the isometry $\tilde{g}|\SS=s_r$ of $\SS$ preserves 
the family $\iota\sp*\RRR(L_{26})\sperp$ of hyperplanes and hence preserves
the tessellation of $\PPP(L_{10})$ by induced chambers.
Since $D$ and $ D^{s_r}$ have the common wall $D\cap (r)\sperp$,
it follows that  $D^{s_r}$ is the induced chamber adjacent to $D$ across the wall $D\cap (r)\sperp$.
For any induced chamber  $D\sprime$,
there exists a chain of induced chambers
\[
D=D\spar{0}, \; D\spar{1}, \;\dots,\; D\spar{m}=D\sprime
\]
such that $D\spar{i-1}$ and $D\spar{i}$ are adjacent for $i=1, \dots, m$.
By induction on the length $m$ of the chain,
we can prove that there exists an isometry $\tilde{g}\sprime\in \OGP{L_{26}}$
preserving  $\iota(\SS)$
such that the induced isometry $\tilde{g}\sprime|\SS$ of $\SS$
maps $D$ to $D\sprime$.
Therefore the tessellation of $\PPP(L_{10})$ by induced chambers is simple.
\begin{table}
\[
\begin{array}{ccllllll}
\texttt{No.}&\texttt{name}&|V_R|&d_{\weyl}&n_{\weyl}&a_r&\texttt{$\ADE$-type of $\Sigma$}&\texttt{numb}\\ 
\hline 
1&\texttt{12A}&256&1&70&1&2A_{1}+D_{8}&10\\ 
 & & & & &8&D_{9}&2\\ 
2&\texttt{12B}&144&2&21&1&2A_{1}+A_{7}&12\\ 
3&\texttt{20A}&160&1&22&1&2A_{1}+D_{4}+D_{5}&10\\ 
 & & & & &4&2D_{5}&6\\ 
 & & & & &5&D_{4}+D_{6}&4\\ 
4&\texttt{20B}&128&1&14&1&2A_{1}+2D_{4}&12\\ 
 & & & & &4&D_{4}+D_{5}&8\\ 
5&\texttt{20C}&192&1&30&2&8A_{1}+A_{3}+D_{6}&15\\ 
 & & & & &6&10A_{1}+D_{7}&5\\ 
6&\texttt{20D}&96&2&15/2&1&2A_{1}+A_{3}+A_{4}&15\\ 
 & & & & &3&A_{4}+D_{4}&5\\ 
7&\texttt{20E}&112&1&10&1&7A_{1}+A_{5}&10\\ 
 & & & & &2&3A_{1}+A_{3}+A_{5}&10\\ 
8&\texttt{20F}&80&2&5&1&2A_{1}+2A_{3}&20\\ 
9&\texttt{40A}&96&1&6&1&6A_{1}+2A_{3}&12\\ 
 & & & & &2&2A_{1}+3A_{3}&12\\ 
 & & & & &3&4A_{1}+A_{3}+D_{4}&16\\ 
10&\texttt{40B}&160&1&16&2&6A_{1}+A_{3}+2D_{4}&16\\ 
 & & & & &4&8A_{1}+D_{4}+D_{5}&24\\ 
11&\texttt{40C}&80&1&4&1&8A_{1}+A_{3}&16\\ 
 & & & & &2&4A_{1}+2A_{3}&24\\ 
12&\texttt{40D}&128&1&10&2&10A_{1}+A_{3}+D_{4}&30\\ 
 & & & & &4&12A_{1}+D_{5}&10\\ 
13&\texttt{40E}&64&2&5/2&1&4A_{1}+2A_{2}&30\\ 
 & & & & &2&2A_{2}+A_{3}&10\\ 
14&\texttt{96A}&64&1&2&1&10A_{1}&32\\ 
 & & & & &2&6A_{1}+A_{3}&64\\ 
15&\texttt{96B}&96&1&4&2&14A_{1}+A_{3}&96\\ 
16&\texttt{96C}&48&2&1&1&6A_{1}&96\\ 
17&\texttt{infty}&32&1&0& & & \\ 
\end{array}
\]
%
\vskip .1cm
\caption{Walls of  $D$}\label{table:walls}
\end{table}
\begin{proof}[Proof of Proposition~\ref{prop:gtilde}]
The isometry $\tilde{g}$ with the hoped-for property is explicitly given
in~\cite{thecompdataL10L26} for each wall of $D$,
and thus the proof of Theorem~\ref{thm:simple} is completed.
\end{proof}
We explain the method by which  we found the isometry $\tilde{g} \in \OGP{L_{26}}$.
It is based on some optimistic heuristics.
The isometry $\tilde{g}$ does not 
necessarily have to satisfy
the conditions (i) and (ii) below.
Fortunately,  
this method worked for every wall $w=D\cap (r)\sperp$ of the induced chamber $D$.
We put
\[
Q:=\set{v\in L_{26}}{\intfLL{v, x}=0\;\;\textrm{for all}\;\; x\in \iota_{\PPP}(w)}.
\]
Since $\dim \,(r)\sperp=9$,
the even lattice $Q$ is negative-definite  of rank $26-9=17$
and contains $R_{\iota}$.
Let $\intfQ{\,,\,}$ denote the symmetric bilinear form of $Q$.
We calculate the set $\RRR(Q)$ of $(-2)$-vectors of $Q$.
The hyperplanes of $Q\tensor\R$
defined by $\intfQ{x, r\sprime}=0$, where $r\sprime \in \RRR(Q)$,
divide $Q\tensor\R$ into
a finite number of regions,
and they correspond bijectively
to the Conway chambers containing $\iota_{\PPP}(w)$.
We put
\begin{equation}\label{eq:SigmaE}
\Sigma:=\shortset{r\sprime\in \RRR(Q)}{\intfLL{\weyl, r\sprime}=1},
\end{equation}
where we regard $\RRR(Q)$ as a subset of $\RRR(L_{26})$ by the embedding  $Q\inj L_{26}$.
Let $P_w$ be an interior point of the wall $w=D\cap (r)\sperp$ in $(r)\sperp$.
Locally around $P_{w}$,
the Conway chamber $C(\weyl)$ is defined by
the inequalities $\intfLL{x, r\sprime}\ge 0$, where $r\sprime$ runs through $\Sigma$.
Let $C(\weyl\sprime)$ be the Conway chamber
defined locally around $P_{w}$
by the opposite inequalities $\intfLL{x, r\sprime}\le 0$,
where $r\sprime$ runs through $\Sigma$.
Then $C(\weyl\sprime)$ is one of the Conway chambers that induce
the induced chamber $D\sprime$ adjacent to $D$ across the wall $w$; $\iota_{\PPP}\inv (C(\weyl\sprime))=D\sprime$.
We search for isometries $\tilde{g}$ of $L_{26}$ such that
\begin{enumerate}[(i)]
\item  the isometry $\tilde{g}$ maps $\Sigma$ to $-\Sigma$, and
\item  the restriction $\tilde{g}|\gen{\Sigma}\sperp$ of $\tilde{g}$
to the orthogonal complement $\gen{\Sigma}\sperp$ in $L_{26}$
of the sublattice $\gen{\Sigma}$ generated by $\Sigma$ is the identity.
\end{enumerate}
If $\tilde{g}$ satisfies (i) and (ii), then
$\tilde{g}$ fixes $\iota_{\PPP}(P_{w}) \in \gen{\Sigma}\sperp\tensor\R $ 
and $\tilde{g}$ maps $C(\weyl)$ to $C(\weyl\sprime)$.
We then check the following conditions:
\begin{enumerate}[(i)]
\setcounter{enumi}{2}
\item  the isometry  $\tilde{g}$ preserves the image of $\iota$, and hence its restriction $\tilde{g}|\SS$ to $\SS$
maps $D=\iota_{\PPP}\inv (C(\weyl))$ to the adjacent chamber $D\sprime=\iota_{\PPP}\inv (C(\weyl\sprime))$,  and
\item  the restriction $\tilde{g}|\SS$ is equal to the reflection $s_r$.
\end{enumerate}
If we find an isometry $\tilde{g}$ satisfying (iii) and (iv),
we are done.
\par
\medskip
{\bf Explanation of Table~\ref{table:walls}.}
We  denote by $d_{\weyl}$ the minimal positive integer such that
$d_{\weyl}\weyl_{\SS}\in \SS$,
where $\weyl_{\SS}$ is the image of $\weyl$ by the orthogonal projection $L_{26}\to \SS\dual$.
We  put $n_{\weyl}:=\intfL{\weyl_{\SS}, \weyl_{\SS}}$.
For a $(-2)$-vector $r$ defining a wall of $D$, we put $a_r:=\intfL{\weyl_{\SS}, r}$.
The root system $\Sigma$ is defined by~\eqref{eq:SigmaE}.
The item \texttt{numb} is the number of walls with the described properties.
\begin{remark}
Let $\phi\colon q(\SS)\isom -q(R_{\iota})$
be the isomorphism induced by $L_{26}\subset \SS\dual\oplus R_{\iota}\dual$,
and  $\OG(\phi)\colon \OG(q(\SS))\isom \OG(q(R_{\iota}))$
the isomorphism induced by $\phi$.
Proposition~\ref{prop:gtilde} can also be proved by
showing that the image of $\eta(\SS)(s_r)\in \OG(q(\SS))$
by $\OG(\phi)$ belongs to the image of $\eta(R_{\iota})\colon \OG(R_{\iota})\to \OG(q(R_{\iota}))$.
\end{remark}
\begin{example}
In~\cite{Ohashi2009},
Ohashi classified all fixed-point free involutions
of the Kummer surface $X:=\Km(\Jac(C))$
associated with the Jacobian variety of a generic genus-$2$ curve $C$,
and showed that $X$ has
exactly $6+15+10$ fixed-point free involutions  modulo conjugation in $\Aut(X)$.
The automorphism group $\Aut(X)$
had been calculated by Kondo~\cite{Kondo1998} by Borcherds' method.
Let $\pi\colon X\to Y$ be the quotient morphism
by a fixed-point free involution of $X$.
We compose  the embedding $\iota_X\colon S_X\inj L_{26}$ used in~\cite{Kondo1998}
with the pull-back homomorphism $\pi^*\colon S_Y(2)\inj S_X$,
and obtain a primitive embedding $\iota_Y\colon S_Y (2) \inj L_{26}$.
We see that
$\iota_Y$ is  of type {\tt 20E}  for $6$ Enriques surfaces,
of type {\tt 40A}  for $15$ Enriques surfaces,
and
of type {\tt 40C}  for $10$ Enriques surfaces.
\par
See~\cite{ShimadaVeniani2019} for more examples,
and for applications to the calculation of $\Aut(Y)$.
\end{example}

\bibliographystyle{plain}

\end{document}